\documentclass[12pt]{amsart}
\usepackage{amssymb}
\usepackage{amsmath}
\usepackage{amsthm}
\usepackage{enumerate}
\usepackage{multirow}
\usepackage[all]{xy}
\usepackage[T1]{fontenc}
\usepackage{mathrsfs}
\usepackage{bbold}
\usepackage{tikz}
\usetikzlibrary{matrix,arrows,backgrounds,shapes.misc,shapes.geometric,patterns,calc,positioning}

\usepackage[colorlinks=true, pdfstartview=FitV, linkcolor=blue, citecolor=blue, urlcolor=blue]{hyperref}
\newtheorem{theorem}{Theorem}

\newtheorem{lemma}[theorem]{Lemma}
\newtheorem{prop}[theorem]{Proposition}
\newtheorem{corollary}[theorem]{Corollary}

\theoremstyle{definition}
\newtheorem{definition}[theorem]{Definition}

\newtheorem{example}[theorem]{Example}

\theoremstyle{remark}
\newtheorem{remark}[theorem]{Remark}

\newcommand{\comment}[1]{}

\newcommand{\key}[1]{\emph{#1}}
\newcommand{\setst}[2]{ \{ #1 \, | \ #2 \} }
\newcommand{\defeq}{\stackrel{\text{\tiny def}}{=}}

\newcommand{\Z}{\mathbb{Z}}
\DeclareMathOperator{\Hom}{Hom}
\DeclareMathOperator{\Ext}{Ext}
\DeclareMathOperator{\Rep}{rep}
\DeclareMathOperator{\rep}{rep}
\DeclareMathOperator{\supp}{supp}

\newcommand{\arrows}{_{1}}
\newcommand{\verts}{_{0}}
\newcommand{\quivq}{{\downarrow\! Q}}
\newcommand{\lin}{L}

\newcommand{\catc}{\mathscr{C}}

\newcommand{\za}{\ensuremath{\alpha}}
\newcommand{\zb}{\ensuremath{\beta}}
\newcommand{\ze}{\ensuremath{\epsilon}}
\newcommand{\zd}{\ensuremath{\delta}}
\newcommand{\zg}{\ensuremath{\gamma}}

\begin{document}
\title{Idempotents in representation rings of quivers }
\author{Ryan Kinser and Ralf Schiffler}\thanks{The second author is supported by the NSF grants  DMS-0908765 and DMS-1001637 and by the University of Connecticut.}
\address{Dept. of Mathematics, University of Connecticut, Storrs, CT 06269}
\email{ryan.kinser@uconn.edu, schiffler@math.uconn.edu}

\maketitle
\begin{abstract}{For an acyclic quiver $Q$, we solve the Clebsch-Gordan problem for the projective representations by computing the multiplicity of a given indecomposable projective in the tensor product of two indecomposable projectives. Motivated by this problem for arbitrary representations, we study idempotents in the representation ring of $Q$ (the free abelian group on the indecomposable representations, with multiplication given by tensor product). We give a general technique for constructing such idempotents and for decomposing the representation ring into a direct product of ideals, utilizing morphisms between quivers and categorical M\"obius inversion.
  }
\end{abstract}
\section{Introduction}\label{sect:intro}

The problem of describing a tensor product of two representations of some algebraic object has appeared in many contexts.  When the category of representations in question has the Krull-Schmidt property (unique decomposition into indecomposables), the problem can be stated for representations $X,Y,Z$ as ``What is the multiplicity of $Z$ as a direct summand in $X \otimes Y$?''
This is sometimes referred to as the \key{Clebsch-Gordan problem}, in honor of A. Clebsch and P. Gordan, who studied the problem for certain Lie groups in the language of invariant theory.

These multiplicities for representations of the groups $SU(2)$ and $SO(3, \mathbb{R})$ give rise to the Clebsch-Gordan coefficients used in quantum mechanics.
In the case of representations of $GL(n,\mathbb{C})$, these multiplicities are the Littlewood-Richardson coefficients, which play an important role in algebraic combinatorics and Schubert calculus \cite{MR1464693}.


Tensor products of quiver representations have been studied by Strassen \cite{strassenasymptotic} in relation to orbit-closure degenerations, and Herschend studied the relation to bialgebra structures on the path algebra in \cite{herschtensorjpaa}.  The Clebsch-Gordan problem for quiver representations is solved explicitly in various situations where a classification of indecomposables is known \cite{MR2563181,herschende6,Herschend:2009kx}, whereas other results on tensor product multiplicities without a classification of indecomposables have appeared in \cite{kinserrank,kinserrootedtrees}.

In this paper, we study the tensor products of representations of a quiver $Q$ in terms of the representation ring $R(Q)$ of the quiver.  This ring has a $\Z$-basis consisting of indecomposable representations of $Q$, with sum corresponding to direct sum and product to tensor product.  The same construction has been used in modular representation theory of finite groups, where it is sometimes called the Green ring \cite{MR842476}.
Besides the actual representations, $R(Q)$ also contains formal additive inverses of representations, and thus ``differences'' of representations.  Understanding the multiplication in this ring can be easier  than directly working with the tensor product of representations. We recall the definition and basic properties of $R(Q)$ in Section \ref{sect:background}.

In Section \ref{sect:projreps}, we solve the Clebsch-Gordan problem for projective representations of an acyclic quiver $Q$ with an explicit formula as follows. 
Let $x,y,w$ be vertices in  $Q$ and $P(x),P(y),P(w)$ be the corresponding indecomposable projective representations. 
\begin{theorem}
The multiplicity of $P(w)$ in $P(x)\otimes P(y)$ equals
\[n_{xw}n_{yw} - \sum_{z\to w} n_{xz}n_{yz},\] where the sum is over all arrows with terminal vertex $w$, and $n_{ij}$ denotes the number of paths from $i$ to $j$ in the quiver.
\end{theorem}

The proof technique is to give an integral change of basis in the subring of $R(Q)$ spanned by projectives, to a new basis consisting of orthogonal idempotents.  
These are trivial to multiply, and then changing back to the original basis gives a multiplication formula for projective representations.
This motivates the construction of other sets of orthogonal idempotents in $R(Q)$.

The projective representations of $Q$ can be concretely presented in terms of discrete data from $Q$, namely, the set of paths in $Q$.  In Section \ref{sect:quivq}, we review a general method for constructing a representation which is not necessarily projective from discrete data, using a morphism of quivers $f\colon Q' \to Q$, also called a coloring of $Q'$ by $Q$, or a quiver over $Q$.
We describe how such a morphism gives rise to a representation of $Q$ via linearization, which generalizes the process of passing from a permutation representation of a finite group to the associated linear representation.
This can be thought of as the opposite course of action to taking a coefficient quiver of a representation \cite{MR1171233}.

Linearization allows us to study certain representations combinatorially from the discrete data in a quiver over $Q$.  A result of Herschend states that, under some mild technical hypotheses, linearization takes the fiber product of two quivers over $Q$ to the tensor product of their linearizations \cite{Herschend:2009kx}.  Thus we expect to be able to analyze the  tensor product of certain representations via quivers over $Q$.

The first main result of the paper, presented in Section \ref{sect main}, is a sufficient condition for a collection of quivers over $Q$ to give rise to a set of orthogonal idempotents in $R(Q)$ (Theorem \ref{thm:main}).  The basic idea is to form an acyclic category (a generalization of a poset) from a collection of quivers over $Q$, then use a categorical form of M\"obius inversion
to orthogonalize the linearizations of these quivers in $R(Q)$.

The motivating application for Theorem \ref{thm:main} is covered in Section \ref{sect:piecat}.  For any acyclic quiver $Q$, we define a category PIE of quivers over $Q$, such that the objects in PIE are in bijection with those indecomposable representations of $Q$ which, after restriction to some subquiver of $Q$, are either projective,  injective, or of dimension 1 at each vertex. We describe morphisms and fiber products in   PIE and show that PIE satisfies the hypotheses of Theorem \ref{thm:main}. 
This allows us to associate an idempotent $e_x\in R(Q)$ to every object $x\in \textup{PIE}$, and to prove our second main result:

\begin{theorem}
Let $Q$ be an acyclic quiver. Then  $R(Q)$ has a direct product structure
\[
R(Q) \cong \prod_{x \in \textup{PIE}_0} \langle e_x \rangle,
\]
where $\langle e_x \rangle $ is the principal ideal generated by $e_x$.
\end{theorem}
  Finally, we present some closed-form expressions for certain values of the M\"obius function of PIE.

\section{Background}\label{sect:background}
A \key{quiver} (or directed graph) is given by $Q=(Q\verts, Q\arrows, s, t)$, where $Q\verts$ is a vertex set, $Q\arrows$ is an arrow set, and $s, t$ are functions from $Q\arrows$ to $Q\verts$ giving the start and terminal vertex of an arrow, respectively.  We assume $Q\verts$ and $Q\arrows$ are finite in this paper.
For any quiver $Q$ and field $K$, there is a category $\Rep_K (Q)$ of representations of $Q$ over $K$.  An object $V=(V_x,\varphi_\za)$ of $\Rep_K (Q)$ is an assignment of a finite dimensional $K$-vector space $V_x$ to each vertex $x \in Q\verts$, and an assignment of a $K$-linear map $\varphi_\za \colon V_{s\za} \to V_{t\za}$ to each arrow $\za \in Q\arrows$.  For any path $p$ in $Q$, we get a $K$-linear map $\varphi_p$ by composition.  Morphisms in $\Rep_K (Q)$ are given by linear maps at each vertex which form commutative diagrams over each arrow; see the book of Assem, Simson, and Skowro\'nski \cite{assemetal} for a precise definition of morphisms, and other fundamentals of quiver representations.  We will fix some arbitrary field $K$ throughout the paper and hence omit it from notation when possible.

There is a natural \key{tensor product} of quiver representations, induced by the tensor product in the category of vector spaces. More precisely, the tensor product of $V=(V_x, \varphi_\za)$ and $W=(W_x, \psi_\za)$ is defined pointwise:  the representation $V \otimes W = (U_x, \rho_\za)$ is given by
\begin{align*}
U_{x} := V_{x} \otimes W_{x} \qquad  &x \in Q\verts \\
\rho_{\za} := \varphi_{\za} \otimes \psi_{\za} \qquad &\za \in Q\arrows .
\end{align*}
It is not difficult to see that $\otimes$ is an additive bifunctor which is commutative and associative, and distributive over $\oplus$ (up to isomorphism).  In other words, this  gives the category $\Rep(Q)$ the structure of a \key{tensor category} in the sense of \cite{DMtannakian}.

The category $\Rep(Q)$ has the \key{Krull-Schmidt property} \cite[Theorem~I.4.10]{assemetal}, meaning that each $V \in \Rep(Q)$ has an essentially unique expression
\[
V \simeq \bigoplus_{i=1}^n V_i
\]
as a direct sum of indecomposable representations $V_i$.  That is, given any other expression $V \simeq \bigoplus \tilde{V}_i$ with each $\tilde{V}_i$ indecomposable, there is a permutation $\sigma$ of $\{1, \cdots n \}$ such that $\tilde{V}_i \simeq V_{\sigma i}$ for all $i$.  Thus the Clebsch-Gordan problem is well defined for $\rep(Q)$.

Since the tensor product distributes over direct sum,  to study $V \otimes W$ we can assume without loss of generality that $V$ and $W$ are indecomposable. A good starting point would then be to have a description of indecomposable objects in $\rep(Q)$.
But a description of all indecomposables is not available for most quivers, so we approach the problem by placing the representations of $Q$ inside a ring $R(Q)$, in which addition corresponds to direct sum and multiplication corresponds to tensor product  (the split Grothenieck ring of $\Rep(Q)$).
Analyzing the properties of $R(Q)$ (e.g. ideals, idempotents, nilpotents) gives a way of stating and approaching problems involving tensor products of quiver representations even in the absence of an explicit description of the isomorphism classes in $\Rep(Q)$.

Let $[V]$ denote the isomorphism class of a representation $V$.  Then define $R(Q)$ to be the free abelian group generated by isomorphism classes of representations of $Q$, modulo the subgroup generated by all $[V \oplus W] - [V] - [W]$.  The operation
\[
[V] \cdot [W] := [V \otimes W] \qquad \text{for }V,\ W \in \Rep(Q)
\]
induces a well-defined multiplication on $R(Q)$, making $R(Q)$ into a commutative ring, called the \key{representation ring} of $Q$.
The Krull-Schmidt property of $\Rep(Q)$ gives that $R(Q)$ is a free $\Z$-module with the indecomposable representations as a basis.
The ring $R(Q)$ generally depends on the base field $K$, but we omit $K$ from the notation since this is fixed in our case.  Also we usually omit the brackets $[\ ]$ and just refer to representations of $Q$ as elements of $R(Q)$.
Although we introduce ``virtual representations'' (those with some negative coefficient in the basis of indecomposables), every element $r \in R(Q)$ can be written as a formal difference
\begin{equation*}
r = V - W \qquad \text{with }V,\ W \in \Rep(Q) .
\end{equation*}
Then any additive (resp. multiplicative) relation $z = x + y$ (resp. $z = xy$) can be rewritten to give some isomorphism of actual representations of $Q$.

\begin{remark}
If one wishes to consider an ideal of relations $I$ for a quiver $Q$, the pointwise tensor product will not generally preserve these relations and thus not be defined for representations of the bound quiver $(Q, I)$.  
However, if $I$ is generated by commutativity relations (that is, relations of the form $p-q$ for paths $p, q$) then the representations of $(Q, I)$ do generate a subring of $R(Q)$.  If $I$ is generated by zero relations (relations of the form $p = 0$ for $p$ a path), then representations of $(Q,I)$ generate an ideal in $R(Q)$ since the tensor product of any map with a zero map is still zero.  The identity element of $R(Q)$ will not satisfy the zero relations, so the ring of representations satisfying $I$ will not generally have an identity element.
Thus, if $I$ consists of zero relations and commutativity relations, we can get a representation ring $R(Q,I)$ without identity.  Throughout the paper, we will not assume that the rings of representations that we work with have identity elements, and thus the term ``subring'' is taken to mean a non-empty subset of a ring which is closed under subtraction and multiplication (and possibly with a different identity element).
\end{remark}

\section{Projective representations}\label{sect:projreps}
Let $Q$ be a quiver without oriented cycles. For every vertex $x\in Q_0$, let $P(x)$ denote the indecomposable projective representation at $x$.  For any two vertices $x,y$, denote by $n_{xy}$ the number of paths from $x$ to $y$ in $Q$. The vector space $P(x)_y$ of the representation $P(x)$ at a vertex $y$ has a basis consisting of all paths from $x$ to $y$; thus $\textup{dim} \,P(x)_y=n_{xy}$.

We will first show in this section that the tensor product of two projective representations is projective, and then we compute the multiplicities $c_{xy}^z$ in the direct sum decompositions
\[P(x)\otimes P(y) =\bigoplus_{z\in Q_0} c_{xy}^z P(z).\]

\begin{lemma}
\label{lem 1}
The tensor product of two projectives is projective.
\end{lemma}

\begin{proof}
Since the tensor product is distributive over the direct sum, it is enough to show the statement for indecomposable projectives. Let $i,j$ be two vertices in $Q$. We need to show that $P(i)\otimes P(j)$ is projective.

We will proceed by induction on the number of vertices in $Q$. If this number is one, then $i=j$, and $P(i)$ is a  representation of dimension one, since $Q$ has no oriented cycles; and thus $P(i)\otimes P(i)=P(i)$ is projective.

Now suppose $Q$ has more than one vertex, and let $i_0$ be a sink in $Q$. If $i=i_0$ then $P(i)$ is the simple representation $S(i)$ and $P(i)\otimes P(j)$ is equal to $P(i)^{\oplus n_{ji}}$; in particular, it is equal to zero if there is no path from $j$ to $i$. This shows that the Lemma holds if $i=i_0$, and a similar argument shows that the Lemma holds if $j=i_0$.

Suppose now that $i$ and $j$ are different from $i_0$. Denote by $Q'$ the quiver obtained from $Q$ by deleting the vertex $i_0$ and all arrows incident to it.
Let $P(i)|_{Q'}$ be the representation of $Q'$ obtained by restricting  to the subquiver $Q'$. Since $i_0$ is a sink in $Q$, we have that $P(i)|_{Q'}$ is a projective $Q'$ representation and therefore the induction hypothesis implies that $P(i)|_{Q'}\otimes P(j)|_{Q'}$ is a projective $Q'$ representation, thus there is an isomorphism 
\[f\ :\ \bigoplus_k c_{ij}^k P_{Q'}(k)\ \longrightarrow \   P(i)|_{Q'}\otimes P(j)|_{Q'},\]
for some $c_{ij}^k\ge 0$ and $P_{Q'}(k)$ the indecomposable projective $Q'$ representation at vertex $k$.
Let $\tilde P=(\tilde P_x,\tilde \varphi_\za)_{i\in Q_0,\za\in Q_1}$ be the corresponding projective $Q$ representation, more precisely, $$\tilde P=\bigoplus_k c_{ij}^k P_{Q}(k).$$

Let us use the notation $P(i)\otimes P(j)=(M_x,\varphi_\za)_{x\in Q_0,\za\in Q_1}$. Then for every vertex $x$, the vector space $M_x$ has a basis consisting of pairs $(c^i,c^j)$, where $ c^i$ is any path from $i$ to $x$ 
and $ c^j$  any path from $j$ to $x$. 
On the other hand, since $i,j$ are both different from $i_0$, the vector space $M_{i_0}$ has a basis consisting of pairs $(c^i\za,c^j\zb)$, where $\za,\zb$ are  arrows with terminal point $i_0$, and $c^i$ is a path from $i$ to $s(\za)$ and $c^j$ is a path from $j$ to $s(\zb)$.
The maps $\varphi_\za$ are given by $\varphi_\za(c^i,c^j)=(c^i\za,c^j\za)$, in particular,
\[\bigoplus_{\za:x\to i_0} \varphi_\za: \bigoplus_{\za:x\to i_0} M_x\to M_{i_0}\] is injective.

The morphism $f$ induces a morphism $\tilde f=(\tilde f_x)_{x\in Q_0} :\tilde P\to P(i)\otimes P(j)$, where $\tilde f_x=f_x$ if $x\ne i_0$, and $\tilde f_{i_0}$ is defined on any path $c\za$, with $\za$ an arrow with $t(\za)=i_0$, as $\tilde f_{i_0}(c\za)=\varphi_\za\tilde f_{s(\za)}(c)$. Clearly, $\tilde f_x$ is an isomorphism for every $x\ne i_0$, and we will show that $\tilde f_{i_0}$ is injective.

Now in the commutative diagram
\[\xymatrix@C90pt@R30pt{  \bigoplus_{\za: t(\za)= i_0} \tilde P_{s(\za)} \ar[r]^{ \bigoplus_{\za: t(\za)= i_0}\tilde\varphi_\za} \ar[d]_{ \bigoplus_{\za: t(\za)= i_0}\tilde f_x} & \tilde P_{i_0} \ar[d]^{\tilde f_{i_0}} \\
 \bigoplus_{\za: t(\za)= i_0} M_{s(\za)} \ar[r]^{ \bigoplus_{\za: t(\za)= i_0}\varphi_\za} & M_{i_0}
}\]
the left column and the top row are isomorphisms, and the bottom row is injective. Therefore the right column $\tilde f_{i_0}$ is injective too.

Thus $\tilde f:\tilde P\to P(i)\otimes P(j)$ is injective with semisimple projective cokernel $P(i_0)^{\oplus t}$ for some integer $t$, and we get a short exact sequence
\[0\to\tilde P\to P(i)\otimes P(j)\to P(i_0)^{\oplus t}\to 0,\]
which splits, since $P(i_0)^{\oplus t}$  is projective. This shows that $P(i)\otimes P(j) $ is projective.  
\end{proof}

The lemma implies that the free abelian group generated by all indecomposable projectives $P(x)$, $x\in Q_0$ has a ring structure whose  addition is given by the direct sum and multiplication by the tensor product (i.e., the projectives span a subring of $R(Q)$). As an additive group, this is isomorphic to $\mathbb{Z}^{Q_0}$ and an isomorphism is given by the Cartan matrix
\[C=\left[n_{xy}\right]_{x,y\in Q_0} = \left[ \underline{\dim}P(1) \cdots \underline{\dim} P(n)\right],\]
where $n=\#Q_0$ and   
$\left[ \underline{\dim}P(1) \cdots \underline{\dim} P(n)\right]$ is the $n\times n$ integer matrix whose $x$-th column is equal to the dimension vector of $P(x)$.  The Cartan matrix is invertible.  Since the dimension vector is multiplicative with respect to the tensor product, this is a ring isomorphism.

 We also have that the $(x,y)$ entry of the transposed inverse matrix $(C^{-1})^t$ can be computed by the formula $\dim \Hom (S(x),S(y))-\dim \Ext(S(x),S(y))$,  see for example
  \cite[III.3.13]{assemetal}. Therefore
 \[ (C^{-1})^t_{x,y}=\left\{ \begin{array}{ll} 1 &\textup{if $x=y$;} \\ - (\textup{number of arrows $x\to y$}) & \textup{if $x\ne y.$}
\end{array}\right.\] 

Let $\ze_x$ denote the standard basis vector $[0,\ldots,0,1,0,\ldots,0]^t$ with $1$ at position $x$, and define $e(x)$ to be the inverse image  of $\ze_x$ under the above isomorphism.
In other words

\[e(x)=\left[ P(1)\cdots P(n)\right] C^{-1} \ze_x ,\]
where $\left[ P(1)\cdots P(n)\right] $ denotes the $1\times n$ matrix whose entries are the indecomposable projective modules, and $C^{-1} \ze_x$ is the $x$-{th}  column of $C^{-1}$.
 
It follows that 
\begin{equation}\label{eq 1.1}
e(x) = P(x)-\sum_{x\to y} P(y)
\end{equation} 
where the sum is over all arrows starting at $x$, and

\begin{equation}\label{eq 1.2} 
P(x)=\sum_z n_{xz} e(z).
\end{equation}

We are now ready to prove the main result of this section.
\begin{theorem}\label{thm 1} Let $x,y\in Q_0$, then
\[P(x)\otimes P(y) =\bigoplus_{w\in Q_0} c_{xy}^w P(w),\]
with $c_{xy}^w=n_{xw}n_{yw} - \sum_{z\to w} n_{xz}n_{yz}$, where the sum is over all arrows with terminal vertex $w$.
\end{theorem}

\begin{proof}
The proof is a simple computation in the representations ring with the orthogonal idempotents $\{e(z)\mid z\in Q_0\}$. 
We have
\[\begin{array}{rcl}
P(x)\otimes P(y) &=& \sum_z n_{xz} e(z) \sum_z n_{yz} e(z)\\
&=&\sum_z n_{xz}n_{yz} e(z) \end{array}\] 
since the $e(z) $ are orthogonal idempotents. 
Now  using equation (\ref{eq 1.1}), we get

\[\begin{array}{rcl}
P(x)\otimes P(y) &=&\sum_z n_{xz}n_{yz}\left(P(z)-\sum_{z\to u} P(u)\right).
\end{array}\]
For a fixed vertex $w$, we can compute $c_{xy}^w$ by collecting terms and get $c_{xy}^w =n_{xw}n_{yw} - \sum_{z\to w} n_{xz}n_{yz}$, where the sum is over all arrows with terminal vertex $w$. This completes the proof. 
\end{proof}

\section{Linearization and M\"obius rings}
\subsection{Quivers over $Q$ and Linearization}\label{sect:quivq}

A \emph{morphism of quivers} $f' \colon Q' \to Q$ sends vertices to vertices and arrows to arrows, and satisfies $s( f'(\za)) = f' (s(\za))$ and $t (f' (\za)) = f' (t(\za))$ 
for each arrow $\za \in Q'\arrows$.
A \key{quiver over $Q$} is a pair $(Q', f')$ where $Q'$ is a quiver, and $f' \colon Q' \to Q$ is a morphism of quivers called the \key{structure map} of $(Q', f')$.  A  \emph{morphism $g$ of quivers over $Q$} is a morphism of quivers which commutes with the structure maps to $Q$:
\begin{equation}\label{eq:quivhom}
\vcenter{\hbox{
\begin{tikzpicture}[>=latex,description/.style={fill=white,inner sep=2pt}] 
\matrix (m) [matrix of math nodes, row sep=3em, 
column sep=2.5em, text height=1.5ex, text depth=0.25ex] 
{  Q' &  & Q''  \\ 
 & Q &  \\ }; 
\path[->,font=\scriptsize] 
(m-1-1) edge node[auto,swap] {$ f' $} (m-2-2)
(m-1-3) edge node[auto] {$ f'' $} (m-2-2)
(m-1-1) edge node[auto] {$ g $} (m-1-3);
$\circlearrowleft$;
\end{tikzpicture}}} .
\end{equation}
So the collection of all quivers over a given $Q$ forms a category denoted by $\quivq$, and we write $g \in \Hom_{\quivq} (Q', Q'')$.

To simplify the notation, we consider the maps $\varphi_\za$ of a representation $V$ to be defined on the total vector space $\bigoplus_{x \in Q\verts} V_x$ by taking $\varphi_\za (V_y) = 0$  when $y \neq s(\za)$.
If $f':Q'\to Q$ is a morphism of quivers then 
the \key{pushforward} $f'_* V=(U_x, \rho_\za) \in \Rep(Q)$ of a representation $V=(V_x, \varphi_\za) \in \Rep(Q')$ is given by
\begin{align}
U_x := \bigoplus_{y \in f'^{-1}(x)} V_y \qquad x \in Q\verts \\
\label{eq:pfmaps}
\rho_\za := \sum_{\zb \in f'^{-1}(\za)} \varphi_\zb \qquad \za \in Q\arrows .
\end{align}

Extending $f'_*$ linearly to $R(Q')$, we get an induced homomorphism $f'_* \colon R(Q') \to R(Q)$ between additive groups, which will not generally be a ring homomorphism.

For a quiver $Q$, we denote by $\mathbb{1}_Q\in \Rep(Q)$ the \key{identity representation} of $Q$: it has a one-dimensional vector space $K$ at each vertex, and the identity map over each arrow.  (The name comes from the fact that this is the identity element of the representation ring $R(Q)$).  When $S \subset Q$ is a subquiver, we can consider $\mathbb{1}_S$ to be a representation of $Q$ via extension by zero: that is, we assign the zero map or vector space to each arrow or vertex outside of $S$.  More generally, we can take any quiver over $Q$ and get a representation of $Q$ by pushing forward the identity representation.  Thus we get a map on objects
\[ \xymatrix@C=30pt@R=0pt{\lin\colon\quivq\ \ar[r]&\Rep(Q) \\
(Q',f') \ar@{|->}[r] & f'_* \mathbb{1}_{Q'}}
\]
which we call the \key{linearization} map.  The representation $f'_* \mathbb{1}_{Q'}$ has a standard basis $\setst{e_x}{x \in Q'\verts}$.
For example, when $(Q',f')$ is the inclusion of a single vertex in $Q$, then its linearization is the simple representation concentrated at that vertex.  When $Q'$ is a quiver of type $A$ with some technical conditions on $f'$, the linearization is a string module.  Similarly, we get a band module or tree module when $Q$ is of type $\tilde{A}$ or when it is a tree, respectively.

\begin{remark}
There is a natural way that one would try to make the linearization functorial: if $g$ is a morphism in $\quivq$ as illustrated in (\ref{eq:quivhom}), one might try to send a standard basis vector $e_x$ of $f'_* \mathbb{1}_{Q'}$ to the vector $e_{g(x)}$ in $f''_* \mathbb{1}_{Q''}$.
However, this will \emph{not} be a morphism of quiver representations, in general.  To see this, one need only take $Q = \bullet \to \bullet$ and consider the map of quivers given by the inclusion of the left vertex.  The corresponding map of vector spaces just described would be a nontrivial morphism from the simple representation of dimension vector $(1,0)$ to the indecomposable of dimension vector $(1,1)$, which is not possible.    By working in some (not necessarily full) subcategory of $\quivq$, one may have some success in making the linearization  functorial (see for example \cite{Crawley-Boevey:1989rr} and \cite[Theorem~18]{kinserrootedtrees}).
\end{remark}


The categorical product of two objects $(Q',f') \times_Q (Q'',f'')$ exists in $\quivq$, which we refer to as the \key{fiber product} of $Q'$ and $Q''$ over $Q$.  It can be realized concretely as having vertex set
\[
(Q' \times_Q Q'')_0 = \setst{(x', x'') \in Q'_0 \times Q''_0}{f'(x') = f''(x'')}
\]
consisting of pairs of vertices lying over the same vertex of $Q$, with an arrow
\[
(x',x'') \xrightarrow{(\za', \za'')} (y',y'')
\]
for each pair of arrows $(x' \xrightarrow{\za'} y', x'' \xrightarrow{\za''} y'') \in Q'_1 \times Q''_1$ such that $f'(\za') = f''(\za'')$.  This common value should be taken as the value of the structure map on the arrow $(\za', \za'')$.


\subsection{Acyclic categories and the M\"obius function}\label{sect:acycliccats}
In order to use an inclusion/exclusion technique to orthogonalize elements of the representation ring, we need a categorial analogue of M\"obius inversion.  This is provided by the work of Haigh \cite{haighmobiusalgebra}, and one may also see the more recent works \cite{MR2393085} and \cite[Ch.~10]{MR2361455}.  We summarize here the tools that we need from this construction.

Following the terminology of Kozlov's book, we call a small category \key{acyclic} if the only endomorphisms are identity morphisms and only identity morphisms are invertible.    This terminology is justified by the observation that if we draw a directed graph whose vertices are the objects and arrows are the morphisms of an acyclic category, then this graph will be acyclic.  For brevity, we denote by $[x,y]_\catc$ the number of morphisms from an object $x$ to an object $y$ in $\catc$.  An acyclic category $\catc$ with finitely many objects $\catc_0$ and morphisms $\catc_1$ admits a \key{M\"obius function}
\begin{equation*}
\mu_\catc \colon \catc_0 \times \catc_0 \to \Z
\end{equation*}
with the following properties:
\begin{align*}
\mu_\catc (x,x) &= \quad 1 \quad \text{for all }x\\
 \sum_{z \in \catc_0} [x,z]_\catc \,\mu_\catc(z,y) &= \left\{ \begin{array}{ll} 0 &    \text{for }x\neq y; 
 \\
 1 &  \text{for }x=y .\end{array}\right.
\end{align*}
We drop the subscripts $\catc$ when this can cause no  confusion.

For example, when $\catc$ is a poset (whose elements are taken to be the objects of $\catc$, and with a unique morphism from $x$ to $y$ if and only if $x \leq y$), we get exactly the classical M\"obius function of the poset \cite[Section~3.7]{stanleyenumcombin}.

For any acyclic category $\catc$, let $H_\catc$ be the  \key{Hom matrix} associated to $\catc$, whose rows and columns are indexed by the objects of $\catc$,
such that the entry $H_{xy}$ in row $x$ and column $y$ is $[x,y]$.  One can choose an ordering of the objects of $\catc$ such that this matrix is upper triangular with ones on the diagonal, since $\catc$ is acyclic, and then one can see from the definition of matrix multiplication that $M\defeq H^{-1}$ will have the value $\mu(x,y)$ in row $x$, column $y$.  

A few facts which will be used frequently  are noted here:
\begin{enumerate}[(a)]
\item From the matrix description we see that 
\[
\sum_{z \in \catc_0} \mu (x,z)[z,y] = 0
\]
 for all $x\neq y$.
\item If $[x,y] = 0$, then $\mu(x,y) = 0$.
\item The value $\mu(x,y)$ can be recursively calculated as
\begin{equation}\label{eq:muformula}
\mu(x,y) = - \sum_{x < z \leq y} [x,z] \mu(z,y)
\end{equation}
where we write $x \leq y$ if there exists a morphism from $x$ to $y$.
\end{enumerate}

\subsection{The M\"obius ring of a finite acyclic category}
The \key{M\"obius ring} $M(\catc)$ of an acyclic category $\catc$ \cite{haighmobiusalgebra} generalizes an object of the same name associated to a poset \cite{greenemobiusalgebra}.  The additive group of $M(\catc)$ is free on the set of objects of $\catc$.  A direct (but somewhat opaque) definition of the product $xy$ of two basis vectors can be given, but we will first give a more computationally useful formulation.  For each object $x$ of $\catc$, define an element
\begin{equation}\label{eq:deltax}
\delta_x \defeq \sum_{z \in \catc_0} \mu(z,x) z
\end{equation}
in $M(\catc)$.  The additive group of $M(\catc)$ is freely generated by $\{\delta_x \}_{x \in \catc_0}$ also, since the Hom matrix and its inverse (which have determinant 1) give the change of basis between this and the defining basis.  Then we just declare these basis elements to be orthogonal idempotents in $M(\catc)$:
\begin{equation}
\delta_x \delta_y =
\begin{cases}
\delta_x 	& \textup{if}\ x= y,\\
0		& \textup{if}\ x \neq y,\\
\end{cases}
\end{equation}
and extend by $\Z$-linearity (so $M(\catc)$ is commutative).  We can recover the original basis elements as
\begin{equation}\label{eq:xfromdelta}
x = \sum_{z \in \catc_0} [z,x] \delta_z ,
\end{equation}
and by substitution the product of two such elements is then
\begin{equation}\label{eq:xymult}
xy = \sum_{z \in \catc_0} \left( \sum_{w \in \catc_0} \mu(z,w)[w,x][w,y] \right) z ,
\end{equation}
recovering the standard definition.

\begin{lemma}
If $x$ is a terminal object for $\catc$ (i.e., each object of $\catc$ has a unique morphism to $x$), then $x$ serves as the identity element of $M(\catc)$.
\end{lemma}
\begin{proof}
If $[w,x] =1$ for all $w \in \catc_0$, the formula (\ref{eq:xymult}) simplifies to
\[
xy = \sum_{z \in \catc_0} \left( \sum_{w \in \catc_0} \mu(z,w)[w,y] \right) z .
\]
The second sum is always 0 unless $z=y$, by fact (a) of the previous subsection, and 1 when $z=y$, thus we have $xy=y$ for all $y \in \catc_0$.
\end{proof}

\begin{remark}\label{rem:cfinite}
The finiteness of $\catc$ can be relaxed in various ways.  For example, the definition (\ref{eq:deltax}) still makes sense if, for each object $x$, there are only finitely many objects $z$ such that $[z,x] \neq 0$.
\end{remark}

\section{Main result on M\"obius rings}\label{sect main}

Let $\catc$ be a full, acyclic subcategory of $\quivq$. From here on, we will always assume that  each object of $\catc$ is a connected quiver over $Q$. Let $L:\catc \to\rep Q$ be the linearization, which we recall is defined only on the objects of $\catc$.  Then $L$ extends by $\Z$-linearity to a map $M(\catc)\to R(Q)$, which we also denote by $L$. In this section, we will show that $L$ is a ring homomorphism when $\catc$ satisfies suitable conditions, and study the image of $L$ in $R(Q)$. We give sufficient conditions on the category $\catc$ so that this subring is isomorphic to the M\"obius ring $M(\catc)$ of the category $\catc$ and construct a basis of idempotents in that case.

We say that the category $\catc$ is \emph{closed under fiber products} if the fiber product of quivers in $\catc$ is a disjoint union of quivers in $\catc$.
We need one more technical condition for linearization to behave well with respect to tensor product.  Following the terminology of \cite{Herschend:2009kx}, we say that a morphism of quivers $f' \colon Q' \to Q$ is a \key{wrapping} if, for every pair of vertices $i',j' \in Q_0'$, the induced map
\[
\{ \textup{arrows from $i'$ to $j'$}\} \xrightarrow{f'} \{ \textup{arrows from $f'(i')$ to $f'(j')$} \}
\]
is injective.  Intuitively, this says that $f'$ does not collapse parallel arrows.  The fiber  product of two wrappings is again a wrapping.

\begin{theorem}
\label{thm:main}
Let $\catc$ be an acyclic subcategory of $\quivq$ whose objects are connected and wrappings, which is closed under fiber products, and  such that for all $x,y\in \catc$, 
\begin{equation}\label{eq 1.3}
\textup{
 $L(x)$ is indecomposable in $\rep Q$ and $L(x) \not\simeq L(y)$ if $x\ne y$.}
\end{equation}
Then the subring of $R(Q)$ generated by $L(\catc)$ is isomorphic to the M\"obius ring $M(\catc)$ of $\catc$.
\end{theorem}
\begin{proof}
The M\"obius ring $M(\catc)$ has the two $\mathbb{Z}$-bases 
\[ \{x\mid x\in \catc\} \quad \textup{and} \quad \{\zd_x=\sum_{z\in \catc_0} \mu(z,x) z \mid x\in \catc\}.
\]
Consider the linearization map 
\[ L :M(\catc)\longrightarrow R(Q), \ x=(Q',f')\mapsto L(x)=f_*'\mathbb{1}_{Q'}.
\]
We will show that $L$ is an injective ring homomorphism. 

The map $L $ is additive by definition, and by condition (\ref{eq 1.3}), $L$ is injective.
In $M(\catc)$ the product is given by $xy=\sum_{z\in \catc_0} [z,x][z,y]  \zd_z$, for $x,y\in \catc$, using the basis of orthogonal idempotents.
Now let  $x\times_Q y=\sqcup_i w_i$ be the decomposition into connected components, where each $w_i\in \catc$.
For a fixed $z$, the set of pairs of maps $\{(z\xrightarrow{f} x,\, z \xrightarrow{g} y)\}$ is in bijection with the set of maps $\bigcup_i \{z\xrightarrow{h} w_i\}$, by the universal property of fiber products and the assumption that elements of $\catc$ are connected quivers.
This implies that $ [z,x][z,y] =\sum_i [z,w_i]$ and so after applying $L$ we have that
\[
L(xy)=
\sum_{z\in \catc_0} \sum_i [z,w_i]  L(\zd_z). \]

On the other hand, $L(x)\otimes L(y) $ is isomorphic to the linearization of $x\times_Q y$, by \cite[Corollary~1]{Herschend:2009kx} (which requires that $x,y$ be wrappings).  In the representation ring $R(Q)$, this gives $L(x)L(y) = \sum_i L(w_i)$.
Now since we already know $L$ is a homomorphism of additive groups, we can use formula \eqref{eq:xfromdelta} to obtain
\[
\sum_i L(w_i) = \sum_i\sum_{z \in \catc_0} [z,w_i]\,L(\zd_z).
\]
This shows that $L $ is a ring homomorphism, and moreover, the image of $L$ is the subring of $R(Q)$ generated by $L(\catc)$, thus it is isomorphic to $M(\catc)$.
\end{proof}

\begin{corollary} Under the assumptions of Theorem \ref{thm:main}, we have the following:
\begin{enumerate}
\item  The subring of $R(Q)$ generated by $L(\catc)$ has a basis of orthogonal idempotents:
\[B=\left\{L(\zd_x)\mid x\in \catc\right\}. \]
\item  When $(Q, id) \in \catc$, this results in a direct product decomposition
\[R(Q) \cong \prod_{x\in\catc}\langle L(\zd_x)\rangle,\]
where $\langle L(\zd_x)\rangle$ is the principal ideal of $R(Q)$ generated by $L(\zd_x)$.
\end{enumerate}
\end{corollary}
\begin{proof}
Statement (1) is immediate from the theorem. Then statement (2) follows because the identity element of $R(Q)$ is the linearization of the identity element $(Q,id)$ of $M(\catc)$, so $1=\sum_x L(\zd_x)$ is a decomposition as a sum of orthogonal idempotents in $R(Q)$.
\end{proof}

\section{The PIE category}\label{sect:piecat}

In Section \ref{sect:projreps} we have seen that the projective representations of an acyclic quiver $Q$ span a subring of $R(Q)$, in which multiplication can more easily be carried out using a basis of orthogonal idempotents.  The duality functor gives a ring isomorphism $R(Q)\cong R(Q^{\rm op})$, so the same can be said for the injective representations of $Q$.  In \cite[\S~4.1]{kinserrootedtrees}, a similar construction is carried out for the collection of idempotent representations of $Q$ (those which are the identity representation of some subquiver).

So the natural question arises as to whether these three sets of idempotents in $R(Q)$ have a common refinement.  That is, we would like to find a subring of $R(Q)$ containing a complete set of orthogonal idempotents which span the set of projective, injective, and idempotent representations.  The first problem one encounters is that the tensor product of a projective with an idempotent representation (which results in the restriction of the projective to a subquiver) is not necessarily projective, injective, or idempotent.  So we need to enlarge the scope of representations that we look at.

\subsection{Subprojective and subinjective representations}
Recall that the \key{support} of a representation $V$ of $Q$, written $\supp V$, is the subquiver of $Q$ consisting of the vertices to which $V$ assigns a nonzero vector space, and the arrows to which $V$ assigns a nonzero map.  For an object $X=(Q',f')$ of $\quivq$, we define $\supp X = f'(Q')$, so that $\supp X = \supp L(X) \subseteq Q$ when $X$ is a wrapping.

\begin{definition}
A representation $V$ of a quiver is called \key{subprojective} or \key{subinjective} if it restricts to a projective or injective representation of its support, respectively.
\end{definition}

To utilize Theorem \ref{thm:main} in the study of tensor products of these representations, we must first present them as linearizations of some quivers over $Q$.


\begin{definition}
A \key{structure quiver} for $V \in \rep(Q)$ is an object $X \in \quivq_0$ such that $L(X) \simeq V$.  A structure quiver $X=(Q', f')$ for $V$ is said to be \key{minimal} if any other structure quiver $Y=(Q'',f'')$ for $V$ has at least as many arrows as $Q'$.
\end{definition}

In the language of   \cite{Ringel:1998gf}, a structure quiver is a ``coefficient quiver'' in some basis.
By dimension reasons, any two structure quivers for a given $V$ have the same number of vertices over each vertex of $Q$.  But the following example shows a basic way that a structure quiver can fail to be minimal.
%
%

\begin{example}\label{ex:nonmin}
Take for our base quiver
\[
Q=\vcenter{\hbox{
\begin{tikzpicture}[point/.style={shape=circle,fill=black,scale=.5pt,outer sep=3pt},>=latex]
   \node[point,label={below:$3$}] (3) at (0,0) {};
   \node[draw, color=white,scale=.6pt,outer sep=2pt] (2) at (2,0) {};
  \node[point,label={below:$2$}] at (2,0) {};
  \node[draw, color=white,scale=.6pt,outer sep=3pt] (1) at (4,0) {};
  \node[point,label={below:$1$}] at (4,0) {};
  
  \path[->]
  	(3.35) edge node[above] {$\za$} (2.145) 
  	(3.-35) edge node[below] {$\zb$} (2.-145)
  	(2) edge node[above] {$\zg$} (1);
   \end{tikzpicture}}}
\]
and consider $P(3)$, the projective representation associated to vertex 3. The ``natural'' structure quiver for $P(3)$ is
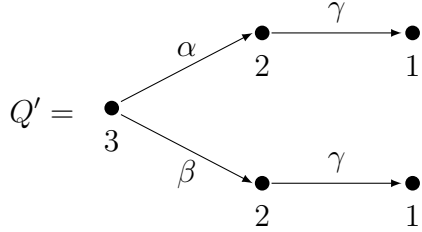
\[
Q'=
\vcenter{\hbox{
\begin{tikzpicture}[point/.style={shape=circle,fill=black,scale=.5pt,outer sep=3pt},>=latex]
   \node[point,label={below:$3$}] (3) at (0,0) {};
   \node[draw, color=white,scale=.6pt,outer sep=2pt] (2a) at (2,1) {};
  \node[point,label={below:$2$}] at (2,1) {};
  \node[draw, color=white,scale=.6pt,outer sep=3pt] (1a) at (4,1) {};
  \node[point,label={below:$1$}] at (4,1) {};
   \node[draw, color=white,scale=.6pt,outer sep=2pt] (2b) at (2,-1) {};
  \node[point,label={below:$2$}] at (2,-1) {};
  \node[draw, color=white,scale=.6pt,outer sep=3pt] (1b) at (4,-1) {};
  \node[point,label={below:$1$}] at (4,-1) {};
  
  \path[->]
  	(3.35) edge node[above] {$\za$} (2a.180) 
  	(3.-35) edge node[below] {$\zb$} (2b.-180)
  	(2a) edge node[above] {$\zg$} (1a)
  	(2b) edge node[above] {$\zg$} (1b);
   \end{tikzpicture}}}
\]
(where we mark the vertices and edges according to what they lie over in $Q$).  But one can quickly see that the linearization of
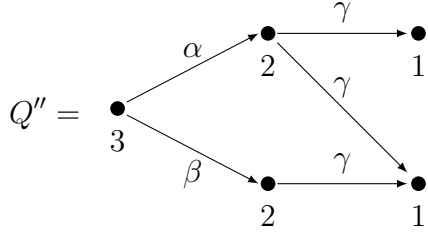
\[
Q''=
\vcenter{\hbox{
\begin{tikzpicture}[point/.style={shape=circle,fill=black,scale=.5pt,outer sep=3pt},>=latex]
   \node[point,label={below:$3$}] (3) at (0,0) {};
   \node[draw, color=white,scale=.6pt,outer sep=2pt] (2a) at (2,1) {};
  \node[point,label={below:$2$}] at (2,1) {};
  \node[draw, color=white,scale=.6pt,outer sep=3pt] (1a) at (4,1) {};
  \node[point,label={below:$1$}] at (4,1) {};
   \node[draw, color=white,scale=.6pt,outer sep=2pt] (2b) at (2,-1) {};
  \node[point,label={below:$2$}] at (2,-1) {};
  \node[draw, color=white,scale=.6pt,outer sep=3pt] (1b) at (4,-1) {};
  \node[point,label={below:$1$}] at (4,-1) {};
  
  \path[->]
  	(3.35) edge node[above] {$\za$} (2a.180) 
  	(3.-35) edge node[below] {$\zb$} (2b.-180)
  	(2a) edge node[above] {$\zg$} (1a)
  	(2b) edge node[above] {$\zg$} (1b)
  	(2a) edge node[above] {$\zg $} (1b);
   \end{tikzpicture}}}
\]
will also give a representation isomorphic to $P(3)$, and that we have an embedding $Q' \subseteq Q''$ as quivers over $Q$.
\end{example}

\subsection{Definition of the PIE category}
We now present the natural structure quivers for subprojective, subinjective, and idempotent representations of an acyclic quiver.  Then we justify calling them ``natural'' by showing that these are the unique minimal structure quivers for these representations.
For each subquiver $T \subseteq Q$, consider the following quivers over $Q$.
\begin{itemize}
\item When $T$ has a unique source $t$, we define the vertex set of the quiver $P_T$ as the set of all paths in $T$ starting at $t$; the structure map as a quiver over $Q$ sends such a path to its endpoint in $Q$.
We put an arrow from the vertex associated to a path $p$ to the one for a path $q$ in $P_T$ exactly when $q$ is obtained by concatenating a single arrow $\za$ onto the end of $p$; in this case, that arrow in $P_T$ is sent to the arrow $\za \in Q_1$ by the structure map.  So in Example \ref{ex:nonmin}, we have $Q'=P_Q$.
In \cite[\S 2]{Enochs:2004ve}, this is called the component of the ``(left) path space'' of $Q$ associated to $t$.
\item When $T$ has a unique sink, $I_T$ is defined dually; its vertex set is the collection of all paths within $T$ that end at the sink.
\item For any subquiver $T\subseteq Q$, the inclusion of $T$ into $Q$ will be denoted by $E_T$ when being considered as a quiver over $Q$.
\end{itemize}

It will always be implicit that $P_T$ or $I_T$ is only defined when $T$ has a unique source or sink, respectively. 

\begin{remark}\label{rem:coincidences}
There are coincidences among the $P$-, $I$-, and $E$-type objects, which we record for reference later.  Two distinct paths are said to be \key{parallel} if they start at the same vertex and end at the same vertex.  Then $E_T = P_T$ if and only if $T$ has a unique source and no parallel paths, while $E_T = I_T$ if and only if $T$ has a unique sink and no parallel paths.  We have $I_T = P_T$ exactly when $T$ is just a single path, in which case we get that these both equal $E_T$ as well.
\end{remark}

\begin{definition} Let
{PIE} be the full subcategory of the category of quivers over $Q$ whose objects are all the $P_T, I_T,$ and $E_T$ as $T$ varies over all subquivers of $Q$.
\end{definition}

\begin{example}\label{ex:pieobjects}
With $Q$ as in Example \ref{ex:nonmin}, the following list gives the distinct objects of PIE.
\begin{itemize}
\item The ten connected subquivers of $Q$.  
\item The $P$-type objects which are not subquivers:
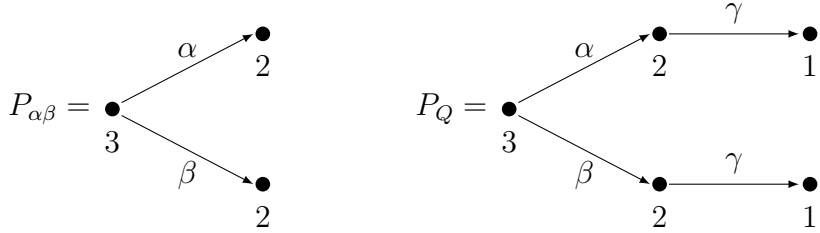
\[
\begin{tikzpicture}[point/.style={shape=circle,fill=black,scale=.5pt,outer sep=3pt},>=latex]
 \node[draw, color=white,label={left:$P_{\za \zb}=$}]  at (0,0) {};
   \node[point,label={below:$3$}] (3) at (0,0) {};
   \node[draw, color=white,scale=.6pt,outer sep=2pt] (2a) at (2,1) {};
  \node[point,label={below:$2$}] at (2,1) {};
   \node[draw, color=white,scale=.6pt,outer sep=2pt] (2b) at (2,-1) {};
  \node[point,label={below:$2$}] at (2,-1) {};
  
  \path[->]
  	(3.35) edge node[above] {$\za$} (2a.180) 
  	(3.-35) edge node[below] {$\zb$} (2b.-180);
   \end{tikzpicture}
\qquad%
\qquad%
\begin{tikzpicture}[point/.style={shape=circle,fill=black,scale=.5pt,outer sep=3pt},>=latex]
    \node[draw, color=white,label={left:$P_{Q}=$}]at (0,0) {};
    \node[point,label={below:$3$}] (3) at (0,0) {};
   \node[draw, color=white,scale=.6pt,outer sep=2pt] (2a) at (2,1) {};
  \node[point,label={below:$2$}] at (2,1) {};
  \node[draw, color=white,scale=.6pt,outer sep=3pt] (1a) at (4,1) {};
  \node[point,label={below:$1$}] at (4,1) {};
   \node[draw, color=white,scale=.6pt,outer sep=2pt] (2b) at (2,-1) {};
  \node[point,label={below:$2$}] at (2,-1) {};
  \node[draw, color=white,scale=.6pt,outer sep=3pt] (1b) at (4,-1) {};
  \node[point,label={below:$1$}] at (4,-1) {};
  
  \path[->]
  	(3.35) edge node[above] {$\za$} (2a.180) 
  	(3.-35) edge node[below] {$\zb$} (2b.-180)
  	(2a) edge node[above] {$\zg$} (1a)
  	(2b) edge node[above] {$\zg$} (1b);
   \end{tikzpicture}
\]
\item The $I$-type objects not included above: 
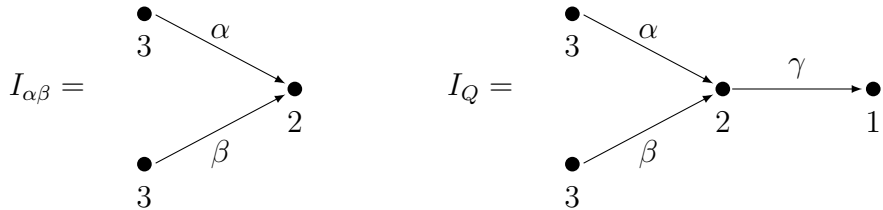
\[
\begin{tikzpicture}[point/.style={shape=circle,fill=black,scale=.5pt,outer sep=3pt},>=latex]
    \node[draw, color=white,label={left:$I_{\za \zb}= $}]at (-0.5,0) {};
   \node[point,label={below:$3$}] (3a) at (0,1) {};
   \node[point,label={below:$3$}] (3b) at (0,-1) {};
   \node[draw, color=white,scale=.6pt,outer sep=2pt] (2) at (2,0) {};
  \node[point,label={below:$2$}] at (2,0) {};
  
  \path[->]
  	(3a.-5) edge node[above] {$\za$} (2.150) 
  	(3b.5) edge node[below] {$\zb$} (2.-150);
   \end{tikzpicture}
\qquad%
\qquad%
\begin{tikzpicture}[point/.style={shape=circle,fill=black,scale=.5pt,outer sep=3pt},>=latex]
    \node[draw, color=white,label={left:$I_{Q}= $}]at (-0.5,0) {};
   \node[point,label={below:$3$}] (3a) at (0,1) {};
   \node[point,label={below:$3$}] (3b) at (0,-1) {};
   \node[draw, color=white,scale=.6pt,outer sep=2pt] (2) at (2,0) {};
  \node[point,label={below:$2$}] at (2,0) {};
  \node[draw, color=white,scale=.6pt,outer sep=3pt] (1) at (4,0) {};
  \node[point,label={below:$1$}] at (4,0) {};
  
  \path[->]
  	(3a.-5) edge node[above] {$\za$} (2.150) 
  	(3b.5) edge node[below] {$\zb$} (2.-150)
	(2) edge node[above] {$\zg$} (1);
   \end{tikzpicture}
\]

\end{itemize}
\end{example}

\begin{lemma}
The objects of PIE are the unique minimal structure quivers for the indecomposable subprojective, subinjective, and idempotent representations.
\end{lemma}
\begin{proof}
It is easy to see that $L(P_T)$ is subprojective, $L(I_T)$ is subinjective, and $L(E_T)$ is idempotent, and that each of these is indecomposable; this is just the standard construction of projectives and injectives which can be found, for example, in \cite[Lemma~III.2.4]{assemetal}.  Thus, we need to show that they are minimal and uniquely so.

If $X=(Q',f')$ is such that $L(X) = L(E_T)$ is an idempotent representation, there is exactly one vertex of $Q'$ over each vertex of $T$. Consequently, all arrows of $Q'$ over a given $\za \in T_1$ must be parallel; taking precisely one arrow over each $\za \in T_1$ is then the unique minimal choice, which is exactly the definition of $E_T$.

Now the $P$-type and $I$-type cases are dual (each follows from the other by working with quivers over $Q^{\rm op}$), so it is enough to prove the statement for the $P$-type case.  Suppose $X = (Q',f')$ is such that $L(X) = L(P_T)$, and fix an arrow $\za \in T_1$.  Then the map $L(X)_\za$ is injective with rank equal to the number of paths in $T$ from the source of $T$ to $s(\za)$, by the description of projectives.  Since a rank $r$ map cannot be the sum of strictly less than $r$ rank one maps, the pushforward construction (\ref{eq:pfmaps}) requires that $Q'$ must have at least this many arrows over $\za$.  So $P_T$ is minimal since it has precisely this many arrows.  

To see that it is unique, we use induction on the number of arrows in $T$.  When $T$ has no arrows the uniqueness is clear.  Now if $T$ has arrows, let $\za$ be an arrow ending at some sink of $T$, and denote by $\tilde{T}$ the connected component of $T\setminus \za$ containing the source of $T$ (i.e., remove $\za$, and if that isolates the vertex $t(\za)$, discard that vertex).  Then working with representations over $\tilde{T}$ (which has a unique source), we define $\tilde{Q}' = f'^{-1}(\tilde{T})$ and see that  the linearization of $\tilde{X} = (\tilde{Q}', f')$ is $L(P_{\tilde{T}})$.

Let $\{v'_1, \dotsc, v'_n \}$ be the vertices of $Q'$ lying over $s(\za)$.  Each $v'_i$ must have at least one outgoing arrow $\za'_i$ in $Q'$ lying over $\za$, because otherwise the vector corresponding to $v'_i$ in $L(X)$ would be in the kernel of the linear map over $\za$, which is not possible since the maps in a projective representation are injective.  By dimension count at the vertex $t(\za)$, each $\za'_i$ ends at a new vertex $w'_i$ of $Q'$ which is not in $\tilde{Q}'$.  By the assumption that $X$ is a minimal structure quiver for $L(P_T)$, we know that $Q'$ has the same number of arrows as $P_T$.  If some $v'_i$ had more than one outgoing arrow over $\za$, that would leave $\tilde{Q}'$ with fewer arrows than $P_{\tilde{T}}$, contradicting the fact that $P_{\tilde{T}}$ is minimal.  So there are exactly $n$ arrows over $\za$ in $Q'$, and $\tilde{Q}'$ has the same number of arrows as $P_{\tilde{T}}$.  By induction, we get that $\tilde{X} = P_{\tilde{T}}$, then the remaining arrows over $\za$ are configured exactly so that $X = P_T$.
\end{proof}

It is worth remarking that we have proven something slightly stronger, namely, that an object of the PIE category actually embeds in any quiver over $Q$ giving the same linearization.

\subsection{Morphisms in PIE} In order to see that the Theorem \ref{thm:main} can be applied to PIE, and eventually do some computations in its M\"obius ring, we need to know the cardinalities of Hom sets.  We first record some simple facts, continuing to use the notation $[X,Y] $ for the cardinality of $\Hom_\quivq (X,Y)$.

\begin{lemma}\label{lem:homtoe}
\begin{enumerate}[(a)] Let $X, Y$ be quivers over $Q$.
\item $[X,Y] = 0$ unless $\supp X \subseteq \supp Y$
\item For $T\subseteq Q$ we have
\begin{equation}\label{eq:homtoe}
[X,E_T] = 
\begin{cases}
1 & \textup{if}\ \supp(X) \subseteq T,\\
0 &	\textup{otherwise.} \\
\end{cases}
\end{equation}
\end{enumerate}
\end{lemma}
\begin{proof}
We can see (a) immediately from the diagram (\ref{eq:quivhom}) in the definition of morphisms in $\quivq$.   Then specializing this diagram to the situation of (b), we see that the dotted line in
\[
\begin{tikzpicture}[>=latex,description/.style={fill=white,inner sep=2pt}] 
\matrix (m) [matrix of math nodes, row sep=3em, 
column sep=2.5em, text height=1.5ex, text depth=0.25ex] 
{  Q' &  & E_T  \\ 
 & Q &  \\ }; 
\path[->,font=\scriptsize] 
(m-1-1) edge node[auto,swap] {$ f' $} (m-2-2);
\path[right hook->,font=\scriptsize] 
(m-1-3) edge node[auto] {$\subseteq$} (m-2-2);
\path[->,dashed,font=\scriptsize]
(m-1-1) edge (m-1-3);
$\circlearrowleft$;
\end{tikzpicture}
\]
can only be filled in when $\supp(X) = f'(Q') \subseteq T$, and only by the morphism $f'$.
\end{proof}

Describing maps to $P$-type objects is slightly more complicated, but we can get enough of a description to count morphism sets in PIE.

\begin{prop}\label{prop:homtop}
Let $T\subseteq Q$ be a subquiver, and $X=(Q',f')$ a quiver over $Q$ with $\supp X \subseteq T$.
\begin{enumerate}[(a)]
\item Given a map of vertex sets $g_0 \colon Q'_0 \to (P_T)_0$ that respects the structure maps to $Q$, there is a unique map of arrow sets $g_1 \colon Q'_1 \to (P_T)_1$ which respects the structure maps to $Q$ and also the start vertex function $s$.
\item The maps in (a) give a morphism $g = (g_0, g_1)\colon Q'\to P_T$ in $\quivq$ if and only if, when regarding the vertices of $P_T$ as paths in $T$, the equation
\begin{equation}\label{eq:homtopcriterion}
g_0 (t(\za')) = g_0 (s(\za')) f'(\za') 
\end{equation}
holds for each arrow $\za' \in Q'_1$. (The operation on the right hand side is concatenation.)
\end{enumerate}
\end{prop}

\begin{proof}
Given a map between vertex sets as in the hypotheses of (a), we explicitly describe the resulting map of arrows.  For each $\za' \in Q'_1$, the arrow $g_1 (\za')$ in $P_T$ must start at $g_0 (s(\za'))$ to respect the $s$ function.  To respect the structure maps to $Q$, this arrow must be labeled with $f'(\za')$.  But in $P_T$,  each vertex has at most one outgoing arrow labeled by a given arrow in $Q$, and the assumption that $\supp X \subseteq T$ guarantees that there is such an arrow for this vertex.
So we can define $g_1(\za')$ as the unique arrow of $P_T$ lying over $f'(\za')$ in $Q$ and satisfying $s(g_1(\za')) = g_0 (s(\za'))$.  This shows (a).

Now suppose that the resulting map is a morphism in $\quivq$.  Then it must respect both the start and terminal vertex function $s,t$, and so an arrow $s(\za') \xrightarrow{\za'} t(\za')$ is sent to
\[
g_0(s(\za')) \xrightarrow{g_1(\za')} g_0 (t(\za'))
\]
in $P_T$, with $g_1(\za')$ lying over $f'(\za')$.  But the construction of $P_T$ is such that this is equivalent to equation (\ref{eq:homtopcriterion}).  Conversely, we need to see that the function $t$ is respected when this equation holds for all arrows.  Since at least $s(g_1(\za')) = g_0 (s(\za'))$, any arrow $s(\za') \xrightarrow{\za'} t(\za')$ is sent to an arrow
\[
g_0(s(\za')) \xrightarrow{g_1(\za')} t(g_1 (\za')) 
\]
in $P_T$.  But then $g_1(\za')$ lying over $f'(\za')$ gives the equation of paths
\[
t(g_1 (\za')) = g_0 (s(\za')) f'(\za') 
\]
by the construction of $P_T$ again, which is exactly equal to $g_0 (t(\za'))$ by assumption.  So $t$ is respected by these maps of vertices and arrows, and thus $g$ is a morphism in $\quivq$.
\end{proof}


\begin{corollary}\label{cor:homptop}
If $Q'$ has a unique source $i'$, then any morphism $g \colon Q' \to P_T$ in $\quivq$ is uniquely determined by $g(i')$.  Consequently, $ [P_S, P_T] $ is equal to 
the number of paths in $T$ from the source of $T$ to the source of $S$
if $S \subseteq T$, and 0 otherwise.
\end{corollary}
\begin{proof}
Part (a) of Proposition \ref{prop:homtop} tells us that the images of arrows under $g$ are determined by the images of the vertices.  Repeated use of equation (\ref{eq:homtopcriterion}) shows that $g(i')$ determines $g(j')$ for any vertex $j'$ lying on a path starting at $i'$. Since $i'$ is the unique source, this determines $g$ completely.

To show the second statement of the corollary, observe first that if $S\nsubseteq T$ then Lemma \ref{lem:homtoe} (a) implies that $[P_S,P_T]=0$. 
Suppose now that $S\subseteq T$.
Compatibility with structure maps requires that any morphism in $\quivq$ sends the source of $P_S$ to a vertex of $P_T$ associated to a path $q$ in $T$ ending at the source of $S$.  Any such choice extends to a morphism $P_S \to P_T$ in the obvious way, by sending a path in $S$ to its concatenation with $q$, which is a path in $T$. Similarly, there is one obvious way to define the map on arrows of $P_S$.  Now the previous paragraph implies that this extension to the rest of $P_S$ is unique.
\end{proof}

\begin{corollary}\label{cor:homietop}
If there   exists a morphism $g\colon Q' \to P_T$ in $\quivq$, then any two arrows with the same terminal vertex in $Q'$ must lie over the same arrow in $Q$.  That is, for $\za',\zb' \in Q'_1$ with $t(\za') = t(\zb')$, we have $f'(\za')=f'(\zb')$.  Consequently, we get that $[E_S , P_T]=0$ unless $E_S = P_S$, and $[I_S, P_T]=0$ unless $I_S = P_S$.
\end{corollary}
\begin{proof}
If there exists such a morphism $g$, we apply equation (\ref{eq:homtopcriterion}) to both $\za'$ and $\zb'$ and then use the assumption that $t(\za')=t(\zb')$ to get
\[
g(s(\za'))f'(\za') = g(t(\za')) = g(t(\zb')) = g(s(\zb')) f'(\zb')
\]
as paths in $Q$.  Since a path can  only end with one arrow, it must be that $f'(\za') = f'(\zb')$.
Now if $E_S$ is distinct from $P_S$, then the subquiver $S$ must either have parallel paths or more than one source.  In either case, there will be two arrows in $E_S$ with the same terminal vertex but different labels, preventing any morphism from $E_S$ to $P_T$.  Similarly, if $I_S$ is distinct from $P_S$, then there are distinct arrows in $I_S$ with the same terminal vertex. 
Thus there can be no morphism from $I_S$ to $P_T$.
\end{proof}

The results of this subsection are summarized Table \ref{fig:piehom}, keeping in mind that by Lemma \ref{lem:homtoe}(a) we need $S \subseteq T$ for any corresponding entry to be nonzero, though we don't write this in each entry of the table.

\begin{table}
\begin{center}
  \begin{tabular}{ |c | c | c | c |}
\hline    from\textbackslash  to & $E_T$ & $P_T$ & $I_T$ \\ \hline
\multirow{2}{*}{$E_S$} 	& \multirow{2}{*}{$1$}  	& 0 unless 	& 0 unless \\ 
 					& 	& $E_S = P_S$ &$E_S = I_S$ \\
    \hline
\multirow{2}{*}{$P_S$} 	&  \multirow{2}{*}{$1$}  	& \# paths in $T$ from	& 0 unless \\ 
 					& 	& source $T$ to source $S$ 	&$P_S = I_S=E_S$ \\
    \hline
\multirow{2}{*}{$I_S$} 	&  \multirow{2}{*}{$1$}  	& 0 unless	& \# paths in $T$ from	 \\ 
 					& 	&$I_S = P_S=E_S$ 	& sink $S$ to sink $T$ 	\\
    \hline
  \end{tabular}
\end{center}
\caption{Summary of morphisms in PIE if $S\subseteq T$}
\label{fig:piehom}
\end{table}

\subsection{Fiber products in PIE}

\begin{lemma}\label{lem:fiberwithe}
For $T\subseteq Q$ and $X = (Q',f')$, we have $E_T \times_Q X \simeq f'^{-1}(T)$.  In other words, fiber product with $E_T$ restricts $X$ to $T$.
\end{lemma}
\begin{proof}
The universal property of the fiber product can be quickly verified: suppose we have a commutative diagram of quiver morphisms given by the solid lines in
\[
\begin{tikzpicture}[>=latex,description/.style={fill=white,inner sep=2pt}] 
\matrix (m) [matrix of math nodes, row sep=3em, 
column sep=2.5em, text height=1.5ex, text depth=0.25ex] 
{Z & & \\ 
   & f'^{-1}(T) &  Q'  \\ 
   &E_T & Q    \\ }; 
\path[dashed,->,font=\scriptsize] 
(m-1-1) edge (m-2-2);
\path[->,font=\scriptsize] 
(m-1-1) edge [bend left=15] node[auto] {$ g $} (m-2-3)
(m-2-3) edge node[auto] {$ f' $} (m-3-3)
(m-2-2) edge node[auto] {$ f' $} (m-3-2)
(m-1-1) edge [bend right=15] node[auto,swap] {$ h$} (m-3-2);
\path[right hook->,font=\scriptsize] 
(m-3-2) edge  (m-3-3)
(m-2-2) edge  (m-2-3);
\end{tikzpicture}
\]
where $Z$ is an arbitrary quiver over $Q$.  We need to see that there is a unique map along the dashed arrow making the diagram commutative everywhere.   The outer square shows that $g(Z) \subseteq f'^{-1}(T)$, so filling in the dashed arrow with $g$ gives a map from $Z$ to $f'^{-1}(T)$ over $Q$ making the two triangles commute.  The upper triangle shows that $g$ is unique.
\end{proof}

We now show that PIE is closed under products with $E$-type objects.   For a vertex $i$ in a quiver $Q$, denote by $\overrightarrow{i}$ the \key{successor closure} of $i$ in $Q$, that is, the full subquiver of $Q$ containing the vertices which can be reached by a path starting at $i$.

\begin{prop}\label{prop:ptimese}
For any $S,T \subseteq Q$, we have that $P_S \times_Q E_T$ is a disjoint union of $P$-type quivers over $Q$.  More specifically, for each source $i$ of $S\cap T$, the quiver $P_{\overrightarrow{i}}$ appears as a component of $P_S \times_Q E_T$ with multiplicity equal to the number of paths from the source of $S$ to $i$ in $S$, where the successor closure is taken inside $S\cap T$.
\end{prop}
\begin{proof}
We know from the previous lemma that $P_S \times_Q E_T$ can be identified with a subquiver of $P_S$ lying over $S \cap T$. 
So the vertices of $P_S\times_Q E_T$ can be identified with paths starting at the source of $S$ and ending in $S\cap T$, with the arrows between them exactly the ones in $P_S$ that lie over $S\cap T$; in particular,  the arrows still fit the description of those in a $P$-type quiver over $Q$.  Now each path ending in $S \cap T$ passes through precisely one source of $S\cap T$, naturally partitioning the vertices as described in the proposition.
\end{proof}

As one would expect, describing the fiber product of an arbitrary $X=(Q',f')$ with $P$-type objects is more complicated.  Roughly, we can think of $X \times_Q P_S$ as a path space for $Q'$ that records only the labels from $Q$ which are traversed to get to a vertex, rather than the exact path.

\begin{prop}\label{prop:ptimesi}
The fiber product of a $P$-type and an $I$-type quiver over $Q$ is a disjoint union of paths in $Q$ (i.e., $E$-type quivers).
\end{prop}
\begin{proof}
Let $S, T \subseteq Q$ be subquivers, so that we want to describe $P_S \times_Q I_T$. By the definition of fiber products, we know that $P_S \times_Q I_T$ has support $S \cap T$, over which $P_S$ and $I_T$ decompose as disjoint unions of $P$-type and $I$-type quivers, respectively.  So if $S\neq T$, we can distribute the product over these disjoint unions and then compute $P_S \times_Q I_T$ from the product of smaller $P$-type and $I$-type quivers.  For each of these products, we can repeat the process until we are left with products over the same subquiver of $Q$ in the base.

Hence we can assume without loss of generality that $S=T=Q$ for the remainder of the proof.  Since, by assumption, $P_Q$ and $I_Q$ are defined, it follows that $Q$ has a unique source $i$ and a unique sink $j$.  Then the vertices of $P_Q \times_Q I_Q$ lying over $k \in Q_0$ are pairs $(p,q)$ consisting of a path $p$ from $i$ to $k$, and  a path $q$ from $k$ to $j$; in other words, each vertex corresponds to a maximal path $pq$ in $Q$ with a distinguished vertex $k$.  Unraveling the definitions, we see that an arrow
\[
(p_1, q_1) \xrightarrow{(a,b)} (p_2, q_2)
\]
in $P_Q \times_Q I_Q$ occurs exactly when $p_1 q_1 = p_2 q_2$ are the same maximal path in $Q$ and $a=b$ is an arrow between adjacent distinguished vertices on this path.  
Thus each connected component of $P_Q \times_Q I_Q$ is a maximal path in $Q$.
\end{proof}

\begin{example}
Continuing with the setup of Examples \ref{ex:nonmin} and \ref{ex:pieobjects}, we get that
\[
\begin{tikzpicture}[point/.style={shape=circle,fill=black,scale=.5pt,outer sep=3pt},>=latex]
    \node[draw, color=white,label={left: $P_Q \times_Q I_Q \simeq $}] (2a) at (-0.5,0.5) {};
   \node[point,label={below:$3$}] (3a) at (0,1) {};
   \node[point,label={below:$3$}] (3b) at (0,0) {};
   \node[draw, color=white,scale=.6pt,outer sep=2pt] (2a) at (2,1) {};
  \node[point,label={below:$2$}] at (2,1) {};
  \node[draw, color=white,scale=.6pt,outer sep=3pt] (1a) at (4,1) {};
  \node[point,label={below:$1$}] at (4,1) {};
   \node[draw, color=white,scale=.6pt,outer sep=2pt] (2b) at (2,0) {};
  \node[point,label={below:$2$}] at (2,0) {};
  \node[draw, color=white,scale=.6pt,outer sep=3pt] (1b) at (4,0) {};
  \node[point,label={below:$1$}] at (4,0) {};
  
  \path[->]
  	(3a) edge node[above] {$a$} (2a) 
  	(3b) edge node[above] {$b$} (2b)
  	(2a) edge node[above] {$c$} (1a)
  	(2b) edge node[above] {$c$} (1b);
   \end{tikzpicture}
\]
can be identified with the two maximal paths in $Q$.
\end{example}

\begin{prop}\label{prop 25}
The fiber product of two $P$-type quivers over $Q$ is a disjoint union of $P$-type quivers.
\end{prop}
\begin{proof}
The same argument as in Proposition \ref{prop:ptimesi} allows us to reduce to the case $P_Q \times_Q P_Q$, where $Q$ has unique source $i$.  Then the vertices of $P_Q \times_Q P_Q$ can be identified with pairs of paths $(p,q)$ that start at $i$ and end at the same vertex of $Q$, and since each vertex of $P_Q$ has at most one incoming arrow, so must each vertex of $P_Q \times_Q P_Q$.

More precisely, an arrow
\[
(p_1, q_1) \xrightarrow{(a,b)} (p_2, q_2)
\]
in $P_Q \times_Q P_Q$ occurs exactly when $a$ and $b$ lie over the same arrow $c$ of $Q$, and both $p_1c=p_2$ and $q_1 c = q_2$ as paths in $Q$; in particular $p_2$ and $q_2$ are parallel paths starting at $i$ that end with the same arrow.  So any pair of paths $(p,q)\in (P_Q\times_QP_Q)_0$ that do not end with the same arrow give a source of $P_Q \times_Q P_Q$, and, for each vertex of the form $(pr,qr)$, where $r$ varies over the paths starting at the common endpoint $j$ of $p$ and $q$, there is a unique path in $P_Q \times_Q P_Q$ starting at $(p,q)$ and ending at $(pr,qr)$.  So in fact $(p,q)$ is the unique source of a connected component of $P_Q \times_Q P_Q$ which is isomorphic to $P_{\overrightarrow{j}}$.  Since all vertices   fall into some connected component of this form (not forgetting the case where both $p$ and $q$ are the trivial path at $i$), we see that $P_Q \times_Q P_Q$ is a disjoint union of $P$-type quivers.
\end{proof}

\subsection{Main result on PIE}
In this subsection, we apply Theorem \ref{thm:main} to the category PIE.

\begin{lemma}\label{lem 26}
For any acyclic quiver $Q$, the corresponding PIE category satisfies the hypotheses of Theorem \ref{thm:main}.
\end{lemma}
\begin{proof}
The category PIE was defined so that the objects are connected and wrappings and linearize to distinct indecomposables. 

To see that PIE is acyclic, we demonstrate an ordering of its objects making the Hom matrix upper triangular unipotent.  First, we ``block'' the objects together into sets $\mathscr{B}_S = \{P_S, I_S, E_S \}$ for each $S \subseteq Q$, keeping in mind our convention of omitting $P_S$ or $I_S$ when the object is undefined, and the possibility of coincidences among $P_S,I_S$ and $E_S$.  If these blocks are ordered so that $\mathscr{B}_S$ comes before $\mathscr{B}_T$ whenever $S \subseteq T$, the Hom matrix will be block lower triangular by Lemma \ref{lem:homtoe}(a).  On the diagonal are then the blocks where $S=T$, which we see from Table \ref{fig:piehom} are always lower triangular: to get a nonzero entry above the main diagonal, we need a coincidence $E_S=P_S$ or $E_S = I_S$, but in this case the corresponding row and column would be omitted as redundant since $S=T$.

The fact that PIE is closed under fiber products follows from applying Lemma \ref{lem:fiberwithe} and Propositions \ref{prop:ptimese}, \ref{prop:ptimesi} and \ref{prop 25} to $Q$ and $Q^{\rm op}$.
\end{proof}

As in section \ref{sect main}, each object $x$ of PIE, defines an idempotent
\begin{equation}
\delta_x \defeq \sum_{z \in \textup{PIE}_0} \mu(z,x) z
\end{equation}
in $M(\textup{PIE})$.  Let $e_x=L(\delta_x)$ be its image in $R(Q)$.  (Note that $e_x$ is different than the $e(x)$ of Section \ref{sect:projreps}.)

We are ready for the main result of this section.
\begin{theorem} Let $Q$ be a quiver without oriented cycles. Then  $R(Q)$ has a direct product structure
\[
R(Q) \cong \prod_{x \in \textup{PIE}_0} \langle e_x \rangle,
\]
where $\langle e_x \rangle $ is the principal ideal generated by $e_x$.
\end{theorem}
\begin{proof}
According to Lemma \ref{lem 26}, Theorem \ref{thm:main} and its corollary apply in this situation. The result now follows.
\end{proof}

\begin{example}
Continuing with the setup of Example \ref{ex:pieobjects}, we can roughly visualize the  PIE category as
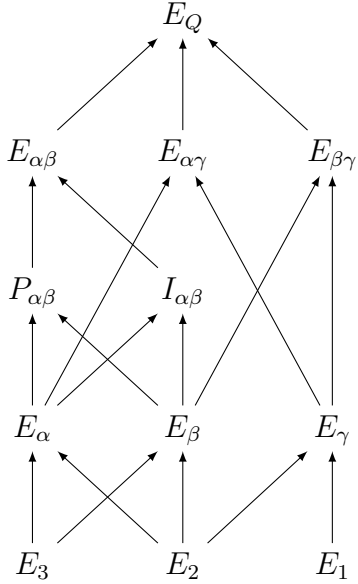
\[
\begin{tikzpicture}[>=latex,description/.style={fill=white,inner sep=2pt}] 
\matrix (m) [matrix of math nodes, row sep=3em, 
column sep=2.5em, text height=1.5ex, text depth=0.25ex] 
{   & E_Q &   \\ 
E_{\za \zb}  & E_{\za \zg} & E_{\zb \zg}  \\ 
P_{\za \zb}  & I_{\za \zb} &   \\  
E_{\za}    & E_{\zb} & E_{\zg}  \\ 
E_3 & E_2 & E_1 \\ }; 
\path[->,font=\scriptsize] 
(m-2-1) edge (m-1-2)
(m-2-2) edge (m-1-2)
(m-2-3) edge (m-1-2)
(m-3-1) edge (m-2-1)
(m-3-2) edge (m-2-1)
(m-4-1) edge (m-3-1)
(m-4-1) edge (m-3-2)
(m-4-1) edge (m-2-2)
(m-4-2) edge (m-3-1)
(m-4-2) edge (m-3-2)
(m-4-2) edge (m-2-3)
(m-4-3) edge (m-2-2)
(m-4-3) edge (m-2-3)
(m-5-1) edge (m-4-1)
(m-5-1) edge (m-4-2)
(m-5-2) edge (m-4-1)
(m-5-2) edge (m-4-2)
(m-5-2) edge (m-4-3)
(m-5-3) edge (m-4-3);
\end{tikzpicture}
\]
Here we have the objects of PIE as nodes, and there is a path from $x$ to $y$ in our diagram if and only if there exists a morphism from $x$ to $y$ in PIE (though we cannot count morphisms from this visualization).  To get the idempotent associated to $x =E_{\za \zb}$, for example, we start by writing
\[
\begin{split}
e_{x} =\ &E_{\za \zb} + \mu(P_{\za \zb}, E_{\za \zb}) P_{\za \zb} + \mu(I_{\za \zb}, E_{\za \zb}) I_{\za \zb} + \mu(E_{\za}, E_{\za \zb}) E_{\za}\\
&+ \mu(E_{\zb}, E_{\za \zb}) E_{\zb}+ \mu(E_3, E_{\za \zb}) E_3 + \mu(E_2, E_{\za \zb}) E_2 ,\\
\end{split}
\]
where we have used the definition of $e_x$, that $\mu(x,x)=1$, and that $\mu(z,x) = 0$ when $[z,x]=0$.	
Then equation (\ref{eq:muformula}) can be used to calculate these coefficients, starting with the ones closest to $E_{\za \zb}$.  For example, we first get 
\[
\mu(P_{\za \zb}, E_{\za \zb} )= \mu(I_{\za \zb}, E_{\za \zb}) = -1
\]
from the fact that $[x,x]=1$.  Similarly, we can find $\mu(E_\za , E_{\za \zb}) = \mu(E_\zb , E_{\za \zb})= 1$.  Then to get $\mu(E_2, E_{\za \zb})$, there is a unique morphism from $E_2$ to each object in the interval between $E_2$ and $E_{\za \zb}$ except $P_{\za \zb}$, for which we have $[E_2, P_{\za \zb}] = 2$.  So here we find
\[
\mu(E_2 , E_{\za \zb}) = -1 - 2(-1) - (-1) - 1 - 1= 0.
\]
A similar computation shows $\mu(E_3, E_{\za \zb}) = 0$, so that finally
\[
e_{x} = E_{\za \zb} - P_{\za \zb} -  I_{\za \zb} + E_{\za} + E_{\zb} .
\]
The entire basis of orthogonal idempotents for $M(\textup{PIE})$ is:
\[
\begin{split}
\{ &E_1,\ E_2,\ E_3,\ 
E_\za - E_2 - E_3,\  
E_\zb- E_2 - E_3,\ 
E_\zg- E_2 - E_1,\ \\
&P_{\za \zb} - E_\za - E_\zb - E_3,\ 
I_{\za \zb}  - E_\za - E_\zb - E_2,\ \\
&E_{\za \zb} - P_{\za \zb} -  I_{\za \zb} + E_{\za} + E_{\zb},\ 
E_{\za \zg} - E_\za - E_\zg - E_2,\  
E_{\zb \zg}- E_\zb - E_\zg - E_2,\ \\
&E_Q -E_{\za \zb} -E_{\za \zg} -E_{\zb \zg} +E_\za  +E_\zb  +E_\zg -E_2
\}.
\end{split}
\]

\end{example}

\subsection{Computation of specific M\"obius functions}
Although one generally cannot expect closed formulas for values of the M\"obius function $\mu$, even in the poset case, we can calculate them for some pairs of objects in the PIE category.  Given two subquivers $S, T \subseteq Q$, we say that they have the same \key{skeleton} if, for every pair of vertices $v,w \in Q_0$, there is at least one edge between $v$ and $w$ in $S$ exactly when there is at least one edge between $v$ and $w$ in $T$.  When $S$ and $T$ have the same skeleton, $P_S$ exists if and only if $P_T$ exists, and similarly for $I$-type objects.

\begin{prop}\label{prop:muvalues}
Let $S \subseteq T$ be subquivers of an acyclic quiver $Q$ which have the same skeleton, and write $\mathcal{A}=T_1 \setminus S_1$ for the set of arrows of $T$ which are not in $S$.  Then the following hold in case $P_S \neq E_S \neq I_S$.
\begin{align}
\label{eq:muespt}
\mu(E_S, P_T) &= 0 \\
\label{eq:mueset}
\mu(E_S, E_T) &= (-1)^{\# \mathcal{A}} \\
\label{eq:mupspt}
\mu(P_S, P_T) &= (-1)^{\# \mathcal{A}} \\
\label{eq:mupset}
\mu(P_S, E_T) &= (-1)^{\# \mathcal{A}+1} \\
\label{eq:mupsit}
\mu(P_S, I_T) &= 0
\end{align}
When $X=P_S = E_S \neq I_S$, we have the following formulas.
\begin{align}
\label{eq:muxet}
\mu(X, E_T) &= 0 \\
\label{eq:muxpt}
\mu(X, P_T) &= (-1)^{\# \mathcal{A}} \\
\label{eq:muxit}
\mu(X, I_T) &= 0
\end{align}
In the case that $Y=P_S = E_S = I_S$, we have 
\begin{align}
\label{eq:muyet}
\mu(Y, E_T) &=  (-1)^{\# \mathcal{A}+1} \\
\label{eq:muypt}
\mu(Y, P_T) &= (-1)^{\# \mathcal{A}} 
\end{align}
Dual formulas also hold (i.e., when $P$- and $I$-type objects are interchanged).
\end{prop}
\begin{proof}
The key is that when $S$ and $T$ have the same skeleton, all the Hom sets involved in finding the formulas of the proposition have at most one element.  In other words, we are computing values of the M\"obius function of some poset in each case. 
For a given $T \subseteq Q$, there is a unique minimal subquiver of $Q$ with the same skeleton as $T$.  Remark \ref{rem:coincidences} implies that this is the only possible subquiver with the same skeleton as $T$ which may be simultaneously $P$- and $E$-type.

Equations (\ref{eq:muespt}) and (\ref{eq:mupsit}) follow from fact (b) of Section \ref{sect:acycliccats}.  
The top row of Table \ref{fig:piehom} shows that the full subcategory of PIE consisting of objects between $E_S$ and $E_T$ (in the Hom order) is isomorphic to the poset of subsets of $\mathcal{A}$.  The M\"obius function of this poset is well known \cite[3.8.3]{stanleyenumcombin}, giving (\ref{eq:mueset}).  
The same argument gives (\ref{eq:mupspt}), since the only objects $Z$ for which there exist morphisms $P_S \to Z \to P_T$ are $P$-type.  
To see (\ref{eq:mupset}), we use equation (\ref{eq:muformula}) to compute
\begin{multline}\label{eq:mucomp1}
-\mu(P_S, E_T) = \sum_{P_S < Z \leq E_T} [P_S, Z] \,\mu(Z, E_T) = \\
\mu(E_S, E_T) + \sum_{S \subsetneq Q' \subseteq T} \left( \mu(E_{Q'}, E_T) + \mu(P_{Q'}, E_T) \right) = \mu(E_S, E_T),
\end{multline}
where the rightmost equality follows from induction by canceling out pairwise each term of the sum.

Now when $X=P_S=E_S$, the equation (\ref{eq:mucomp1}) still holds except that the term $\mu(E_S, E_T)$ is absent, so we get (\ref{eq:muxet}).  
Again, morphisms $X \to P_T$ can only factor through $P$-type objects, so the same argument for (\ref{eq:mupspt}) applies to give (\ref{eq:muxpt}).
In this case there are still no morphisms from $P_S$ to $I_T$, so (\ref{eq:muxit}) follows.

Finally, when $Y=P_S=E_S=I_S$ is just a path in $Q$, it has morphisms to objects of all types in PIE.  So we get
\begin{equation}\label{eq:mucomp2}
\begin{split}
-\mu(Y, E_T) = \sum_{P_S < Z \leq E_T} [Y, Z] \mu(Z, E_T) = \\
\sum_{S \subsetneq Q' \subseteq T} \left( \mu(E_{Q'}, E_T) + \mu(P_{Q'}, E_T) +\mu(I_{Q'}, E_T) \right) = \\
\sum_{S \subsetneq Q' \subseteq T} \mu(P_{Q'}, E_T) = -\sum_{S \subsetneq Q' \subseteq T} \mu(E_{Q'}, E_T) = - (-1)^{\# \mathcal{A}} = (-1)^{\# \mathcal{A}+1}
\end{split}
\end{equation}
by applying formulas from the first group and canceling some terms.  
The same argument for (\ref{eq:mupspt}) and (\ref{eq:muxpt}) will give (\ref{eq:muypt}).  By applying the formulas to $Q^{\rm op}$, we get similar formulas on $Q$ with $P$- and $I$-type objects interchanged.
\end{proof}

The hypothesis that $S$ and $T$ have the same skeleton can be relaxed for several of the formulas; for example, the same proof shows that (\ref{eq:muespt}) and (\ref{eq:mupsit}) hold for all subquivers $S$ and $T$ when $P_S \neq E_S$.  

\section{Future Directions}
Here we suggest a few directions for future work.

\begin{enumerate}[1)]
\addtolength{\itemsep}{0.5\baselineskip}
\item What are other examples of categories of quivers over $Q$ satisfying the hypotheses of the Theorem \ref{thm:main}?  For example, when $Q$ is any quiver, Section 4 of \cite{Herschend:2009kx} gives such a category (with infinitely many objects, but see Remark \ref{rem:cfinite}) in the course of studying string and band modules.  Or when $Q$ is a rooted tree quiver, there is a collection of ``reduced quivers over $Q$'' given in \cite{kinserrootedtrees} which satisfies these hypotheses.

A result of Ringel states that if $V$ is an exceptional representation of a quiver (i.e., $\Ext^i (V,V) = 0$, for all $i\ge1$), then $V$ has a structure quiver which is a tree \cite{Ringel:1998gf}.  This structure quiver is not unique, but one may try to give ``good'' choices of structure quivers for some class of exceptional modules so that Theorem \ref{thm:main} can be applied.

\item Can we get more closed formulas for values of $\mu$, in addition to Proposition \ref{prop:muvalues} (for the PIE category, or any other example)?

\item When does Theorem \ref{thm:main} give \emph{all} of the idempotents of $R(Q)$ (or how can it be improved to give all idempotents)?  That is, under what conditions on $\catc$ is it impossible to write each $L(\zd_x)$ as a nontrivial sum of idempotents?  The PIE category will not generally give all idempotents, but the rooted tree case mentioned above does.

\item Is there a representation theoretic interpretation for the idempotents obtained from the PIE category?  For example, given $x \in \textup{PIE}_0$, what properties of $V \in \rep(Q)$ are necessary or sufficient for $e_x V=0$? (cf. Prop. 32 and 35 of \cite{kinserrootedtrees})
\end{enumerate}

\subsection*{Acknowledgement} The authors would like to thank the referee for helpful comments.

\bibliographystyle{alpha}
\bibliography{ryanbiblio}

\end{document}